\documentclass[11pt]{amsart}

\usepackage{amsthm}
\usepackage{latexsym}
\usepackage{amsmath}
\usepackage{latexsym}
\usepackage{amssymb,amscd}
\usepackage{amsfonts}
\usepackage{xspace}
\usepackage[frame,ps,matrix,arrow,curve,rotate]{xy}
\usepackage{amsfonts}
\usepackage{mathrsfs}
\usepackage{diagrams}

\newtheorem{theorem}{Theorem}[section]
\newtheorem{proposition}[theorem]{Proposition}
\newtheorem{conjecture}[theorem]{Conjecture}

\newtheorem{lemma}[theorem]{Lemma}
\newtheorem{corollary}[theorem]{Corollary}

\theoremstyle{remark}
\newtheorem{remark}[theorem]{\bf Remark}
\newtheorem{example}[theorem]{\bf Example}
\theoremstyle{definition}

\def\P{{\bf P}}
\def\C{{\bf C}}

\def\O{{\mathcal O}}
\def\cO{{\mathcal O}}
\def\cI{{\mathcal I}}

\def\cE{{\mathcal E}}

\def\R{{\bf R}}

\def\Hilb{{\rm Hilb}}

\def\length{\mathop{\rm length}}
\def\codim{\mathop{\rm codim}}

\def\Hom{\mathop{\rm Hom}}

\def\Hom{{\rm Hom}}
\def\Spec{\mathop{\rm Spec}}

\def\coker{\mathop{\rm coker}}

\def\red{{\rm red}}
\def\PP{{\P}}
\def\gm{{\mathfrak m}}
\def\reg{{\rm reg}}

\def\Der{{\rm Der}}

\begin{document}

\title{Fibers of Generic Projections}
\author{Roya Beheshti and David Eisenbud
}

\begin{abstract}
Let $X$ be a smooth projective variety of dimension $n$ 
in $\P^r$, and let $\pi: X\to \P^{n+c}$
be a
general linear projection ,
with $c>0$. In this paper we bound the
scheme-theoretic complexity of the fibers of $\pi$.

In his famous work on stable mappings, John
Mather extended the classical results by showing
that the number of distinct points in the fiber is bounded
by $B:=n/c+1$, and that, when $n$ is not too large,
the degree of the fiber (taking the
scheme structure into account) is also bounded by $B$. A result of Lazarsfeld shows that this 
fails dramatically for $n\gg 0$. We describe a new invariant of the scheme-theoretic
fiber that agrees
with the degree in many cases and is always bounded by $B$.
We deduce, for example, that if we write a fiber 
as the disjoint union of schemes
$Y'$ and $Y''$ such that $Y'$ is the union of the locally complete intersection
components of $Y$, then
$\deg\ Y'+\deg\ Y''_\red\leq B$.
  
Our method also gives a sharp bound on the subvariety of $\P^r$ 
swept out by the $l$-secant lines of $X$ for any positive integer 
$l$, and we discuss a corresponding bound for highly secant
linear spaces of higher dimension. These results extend Ziv Ran's
``Dimension+2 Secant Lemma'' \cite{ran}.
\end{abstract}

\maketitle\nopagebreak
\vspace{0.5cm}

\section{Introduction}
\footnotetext[1]{Both
authors are grateful to MSRI, where most of the work of this paper
was done while the first author was a postdoctoral fellow. The second
author was partially supported by the NSF grant DMS 0701580.}

Throughout this paper, we work over an algebraically closed field $k$ of characteristic zero.
We denote by $X \subset  \PP^r= \PP^r_{k}$ a smooth projective variety, and by
$\pi:X\to \PP^{n+c}$ a general linear projection with $c>0$.

We are interested
in how large and complex the fibers $\pi^{-1}(q)$ can be for $q\in \PP^{n+c}$. 
Classical computations, greatly extended by work of John Mather, show that
when $n$ is small (for example $c=1$ and $n\leq 14$) 
the degree of any fiber is bounded by $n/c+1$. For $n$ large, however,
Lazarsfeld has shown that the fibers can have exponentially
greater degree (for example with $c=1$ and $n= 56$ there will sometimes be fibers of
degree $\geq 70$.) In fact, when $n$ is large compared to $c$  we know 
no bound on the degrees of the fibers
that depends on $n$ and $c$ alone. Nevertheless, the degree of a fiber is
bounded by $n/c+1$ in so many cases that it is tempting to repair the situation by
looking for ways to replace the degree by some other locally defined invariant, one that
``often'' agrees with the degree and always takes values $\leq n/c+1$. In this 
paper we introduce an invariant that provides just such a replacement.

Any fiber of $\pi:X\to \PP^{n+c}$ can be expressed as the scheme-theoretic
intersection of $X$
with a linear space $\Lambda$. Our invariant is defined more generally
for the intersection of two schemes $X,Y$ in an ambient scheme $P$. 

We first define a coherent sheaf $Q(X,Y)$, 
supported on $Z:=X\cap Y$, as the cokernel of the restriction map
$$
Q(X,Y):= \coker(\Hom(\cI_{Z/X}/\cI_{Z/X}^2, \cO_Z) \to \Hom(\cI_{Y/P}/\cI_{Y/P}^2, \cO_Z)).
$$
When the intersection is ``too small'' in the sense that $\codim Y - \dim X > \dim Z$, then we define

$$ 
q(X,Y) := \frac{ \deg Q(X,Y)}{\codim Y - \dim X -\dim Z}.
$$
The invariant $q(X,Y)$ is often equal to the degree of $Z$. We will show that this is the case,
for example, when  $X\cap Y$ is locally a complete intersection, or, more generally, is a smooth point of
the ``smoothing component'' of its Hilbert scheme in $Y$
or $X$. (Curiously, though much attention has been paid to computing intersection numbers in the
case of excess intersection, where $\codim Y - \dim X < \dim Z$, we are not aware of nontrivial invariants
other than $\deg Z$ for the case $\codim Y - \dim X > \dim Z$.)

Under good circumstances $Q(X,Y)$
is the module of obstructions to an infinitesimal flat deformation of $Y$ inducing a flat deformation
of $X\cap Y$.
In general, it measures the 
``excess intersection'' of $X$ and $Y$: when these are both
locally complete intersections, $Q(X,Y)$ vanishes if and only if the
intersection is dimensionally transverse. If $X,Y$ and $P$ are smooth and $Z=X\cap Y$ is finite, then 
$Q(X,Y)$ is equal to $Q(Y,X)$, and up to a direct sum with
a free $\cO_Z$ module, these depend only on $Z$ 
(see Theorem \ref{di} and Proposition \ref{independence}). From this one can show that
$q(X,Y)$ depends only on $Z$ and the number $\codim Y - \dim X -\dim Z$.

We now return to the case of a generic projection $\pi:X\to\PP^{n+c}$.
A fiber $Z$ can be written as $X\cap \Lambda$, where
$$
\codim \Lambda -\dim X = c > 0 = \dim Z,
$$
so $q(X,\Lambda)$ is defined and, as we have noted, it is equal to the degree of $Z$
in many cases.
Our main result says that $q(X,\Lambda)$
behaves as we would hope:

\begin{theorem}\label{main1}
If $X$ is a smooth projective variety of dimension $n$ in $\P^r$, 
and if $\pi: X \to \P^{n+c}$ is a general projection, then 
every fiber $X\cap \Lambda$, where $\Lambda$ is a linear
subspace containing the projection center in codimension 1, satisfies;

$$
q(X,\Lambda)\leq \frac{n}{c}+1.
$$
\end{theorem}

The invariant $q(X,\Lambda)$ can be greater than or less than the degree of the
scheme $X\cap \Lambda$. Locally
at a point of $X\cap \Lambda$ that is not a complete intersection,
we can show in the situation above
that $q(X,\Lambda)$ is at least $\frac{2}{c}+1$ (see also Conjecture \ref{Q-reg}.)

A corollary of the work of John Mather \cite{mat1} is that a fiber of $\pi$ as in 
Theorem \ref{main1} can never have more than $n/c+1$ distinct points. 
Mather's work, and also that of Ziv Ran \cite{ran} implies that if a fiber is
curvilinear, then it has degree at most $n/c+1$. Both these statements
are improved by the following consequence of Theorems \ref{main1}
and \ref{qlength}. The class of ``licci'' schemes (schemes in the 
linkage class of complete intersections), which is explained in Section 3,
contains, for example, the class of complete intersections, but also
all zero-dimensional schemes of embedding dimension at most 2.

\begin{corollary}\label{first cor}
Suppose that $X$ is a smooth $n$ dimensional
variety in $\P^r$ and that $\pi$ is a generic projection
of $X$ into $\P^{n+c}$, with $c\geq 1$. Write
$\pi^{-1}(p)=X\cap \Lambda$, where $\Lambda$ is a linear space 
containing the projection center in codimension 1. 
If we decompose $\pi^{-1}(p)$ as $Y\cup Y'$ where $Y$ is the
union of all the licci components,
then 
$$
\deg Y+(1+\frac{3}{c})\deg Y'_\red \leq q(X,\Lambda)\leq n/c+1.
$$
In case $c=1$ this may be improved to 
$$
\deg Y+5\deg Y'_\red \leq  q(X,\Lambda)\leq n+1.
$$

\end{corollary}

\subsection*{Castelnuovo-Mumford Regularity}
It would be nice to have bounds on the complexity of the fiber in terms
of more familiar invariants. One attractive
possibility not contradicted by Lazarsfeld's examples is given by the following
conjecture:

\begin{conjecture}\label{fiber-reg}
Let $X\subset \PP^r$ be a smooth projective variety of dimension $n$, and
let $\pi: X\to \PP^{n+c}$ be
a general linear projection. If $Z\subset X\subset \PP^r$
is any fiber, then the Castelnuovo-Mumford regularity of
$Z$, as a subscheme of $\PP^r$ is $\leq n/c +1$.
\end{conjecture}

If true, this conjecture is sharp: an argument of
Lazarsfeld, reproduced in Proposition \ref{lazarsfeld argument}
shows that the generic
projections to $\PP^{n+1}$ of codimension 2, arithmetically 
Cohen-Macaulay varieties of dimension $n$
have fibers with degree $n+1$ as long as the varieties don't lie on 
hypersurfaces of degree $\leq n$; since such fibers are
automatically colinear, their regularity is $n+1$. 

The conjecture is also
sharp for non-degenerate surfaces in $\PP^5$.
The Fano embedding of the Reye congruence (Example \ref{enriques}) is a nondegenerate
smooth Enriques surface in $\PP^5$ 
 whose generic projection to $\PP^3$
has fibers with three collinear points, and thus regularity 3.

Conjecture \ref{fiber-reg} would follow from Theorem \ref{main1} and 
the following conjectural comparison
between the invariant $q$ and the regularity of the intersection:

\begin{conjecture}\label{Q-reg}
Let $X\subset \PP^r$ be a smooth projective variety of dimension $n$, and
let $\pi: X\to \PP^{n+c}$ be
a general linear projection. If $Z\subset X\subset \PP^r$
is any fiber, written as
$Z=X\cap \Lambda$, where $\Lambda$ is a linear
subspace containing the projection center in codimension 1,
satisfies 
$$
\reg\ Z\leq q(X,\Lambda). 
$$
\end{conjecture}

If Conjecture \ref{fiber-reg} is true, it could be used to give
a new bound on the regularity of smooth varieties. 
The Eisenbud-Goto conjecture \cite{eisenbud}
 predicts that for any integral non-degenerate variety $X$ in $\P^r$, the Castelnuouvo-Mumford regularity of $X$ is less than or equal to $\deg (X) - \rm{codim} (X) +1$.  This conjecture 
has been verified for curves and smooth surfaces (see \cite{GLP} and \cite{laz1}). 
The best known bound on the regularity of a smooth variety $X$
of dimension $n$ is given by Bertram-Ein-Lazarsfeld
(Cor. 2.1 in \cite{BEL}),
extending a theorem of Mumford:
$$ {\rm{reg}} (X) \leq \min{\{\codim X, 1+\dim X\}}  (\deg X -1) - 1.$$
Kwak \cite{kwak} has improved this for smooth varieties of dimension 
$\leq 6$ to a bound of the form $\deg X-\codim X+{\rm constant}$. Using projection methods of \cite{laz1} and \cite{kwak}, 
Conjecture \ref{fiber-reg} would imply that for any smooth non-degenerate projective variety of dimension 
$n$ in $\P^r$, 
$$ 
\reg(X) \leq \deg (X) - \codim (X) + 1 + c_{n,r} 
$$
where $c_{n,r} = \sum_{i=3}^{n+1} (i-2) {{r-n-2+i}\choose{i}}.$

\subsection*{Secant Lines and Planes}
We can also use $Q(X,\Lambda)$
to give bounds on the dimension of the subvariety of 
$\P^r$ swept out by $l$-secant lines of $X$. An $l$-secant line of 
$X$ is a line in $\P^r$ whose intersection with $X$ has degree at least 
$l$ or is contained in $X$. With $X$ as above, let $S_{l}$ be 
the subvariety of $\P^r$ swept out by the $l$-secant lines of $X$. 
Ziv Ran's celebrated ``Dimension+2 Secant Lemma'' \cite{ran} says that the dimension of
$S_{n+2}$ is at most $n+1$. We give a more conceptual proof,
which allows us to go further and get a sharp bound on the dimension of the 
subvariety swept out by the $l$-secant lines of $X$ for any $l$.

\begin{theorem}\label{main3}
Let $X$ be a smooth projective variety of dimension $n$ in $\P^r$, 
and let $S_{l}$ be the subvariety of $\P^r$ swept out by all the 
$l$-secant lines of $X$. If $l\geq 2$ then 
$$\dim S_{l} \leq  \frac{nl}{l-1}+1.$$
\end{theorem}  

An argument of Lazarsfeld (Proposition \ref{lazarsfeld argument})
shows that this bound is achieved in the case
$l=n+1$ for arithmetically Cohen-Macaulay
 varieties of codimension 2.
In Theorem \ref{plane} we give a corresponding bound for secant planes, as well.

\begin{remark}[Length and Degree]
We generally use the word \emph{degree} when speaking of schemes of dimension zero,
and \emph{ length} when speaking of modules of dimension zero. 
Since we are working over an algebraically closed field, the degree $\deg Z$
of a scheme of dimension zero is the same as the \emph{length} of the module
$\cO_Z$.
\end{remark}

We are grateful to Craig Huneke, who first showed us how to prove that
$Q(X,Y)$ is nonzero in the setting of Theorem \ref{qlength}, and to Joe Harris and
Sorin Popescu, who helped us with interesting examples. Rob Lazarsfeld
has shared numerous insights. In particular he pointed out to us
the bad behavior of fibers in high dimensions, and explained how to prove
that codimension 2 complete intersections nearly always have fibers of
the conjecturally maximal regularity.

Many explicit computations informed our intuition about the subject of this paper.
These would not have been possible without the help of Macaulay2 and its
makers, Dan Grayson and Mike Stillman \cite{M2}. We also used the facilities of Magma,
and we are particularly grateful to Allan Steel, of the Magma group,
for his generous help in writing and running these programs.

\section{Mather's Upper Bounds and Lazarsfeld's Examples}

Mather's trasversality theorem holds for multi-jet spaces and a class of 
their subvarieties that are called modular. Mather showed that if $X$ is an $n$-dimensional 
smooth projective variety, then the multi-jet of a general linear projection 
$f: X \to \P^{n+c}$ is transverse to all the modular subvarieties of the
multi-jet spaces $_sJ^k(X, \P^{n+c})$, and that all the Thom-Boardman strata in multi-jet spaces 
are modular (see the last sentence of \cite{mat1}). He used these results in \cite{mat2} to conclude that in the nice range of dimensions 
($n < \frac{6}{7}(n+c)+\frac{8}{7}$, or $n < \frac{6}{7} (n+c) + \frac{9}{7}$ and $c \leq 3$), 
a general projection is $C^{\infty}$ stable, and this enabled him to describe 
local normal forms for general projections in the nice range of dimensions.  In particular he
proved that when $c=1$ and $n \leq 14$, the degrees of fibers of a general
projection is bounded by $n+1$, just as in the classical cases.

The following result is a consequence of Mather's transversality theorems  \cite{mat1} and
the formula proved by Boardman \cite{boardman} for the codimensions of the Thom-Boardman
strata.  We say that a map $\pi:X\to Y$ between smooth spaces
 has \emph{tangential corank $d$} at a point $p\in X$
if the induced map on tangent spaces has rank $=\dim X-d$ at $p$.

\begin{theorem}[Mather]\label{mather}
Let $X\subset \PP^r$ be a smooth projective variety of dimension $n$ and 
let $\pi:X\to \PP^{n+c}$ be a general linear projection, with $c>0$.
If the fiber $\pi^{-1}(q)$ contains points $p_1,\dots,p_d$
and $\pi$ has corank $d_i$ at $p_i$, then
$$
\sum_{i=1}^d( \frac{d_i^2}{c} + d_i+ 1)
\leq \frac{n}{ c}+1.
$$
\end{theorem}

For example, the number of
distinct points in the fiber of a generic projection
is at most $n/c+1$. Mather's theory also implies
that any curvilinear fiber
(that is, one that can be embedded in a smooth curve)
has degree $\leq n/c+1$.
 (Roberts \cite{roberts} proves these results
by a more elementary method, but only
for embeddings of varieties that have been composed with sufficiently
high Veronese maps.) But it is easy to see
that knowing the Thom-Boardman symbol of the projection
at every point is not enough to bound the degree of the fiber
except in the curvilinear case.

In fact, Lazarsfeld shows in \cite{laz} (Volume 2, Prop. 7.2.17)
that for if $X\subset \PP^r$ is a sufficiently ample embedding
of a smooth variety of dimension $n$,
then a general projection of $X$ to $\PP^{n+1}$
will have points of corank $d$ whenever 
$d(d+1) \leq n$. He used this to show that the image of the
projection would have points of mulitplicity at least 
on the order of $2^{\lfloor \sqrt n\rfloor}$. The co-rank $d$ condition
also gives a lower bound on the lengths of the fibers:

\begin{proposition}
If $X$ is a smooth projective variety of dimension $n$ and $p\in X$ is a point of corank $d$ for the projection map 
$\pi: X \to \P^{n+1}$, then the degree of the fiber $\pi^{-1}\pi(p)$ is at least ${{d+1}\choose{\lceil \frac{d}{2} \rceil}}$.  
\end{proposition}

\begin{proof}
Let $\hat{\O}_{X,p}$ be the completion of the local ring of $X$ at $p$, and let $\mathfrak{m}$ be its maximal ideal. Let 
$f_1, \dots, f_{n+1} \in \hat{O}_{X,p}$ be the functions locally defining $\pi$ at $p$. Since $p$ is a point of corank $d$, we can assume 
that $n-d$ of the $f_i$, say $f_{d+2},\dots,f_{n+1}$, form a regular system of parameters, so
 that $R := \hat{\O}_{X,p} / (f_{d+2},\dots, f_{n+1})$ is a power series ring in $d$ variables, 
while
 $f_1, \dots, f_{d+1} \in \mathfrak{m}^2$.
 Set $I = (f_1, \dots, f_{d+1})$.

Setting $S=R[[y]]$, we can write $R/I=S/J$ where $J = (y) + IS$. 
Denote by $\mathfrak{p}$ the maximal ideal of $S$, and
let $f_i'=f_i+q_i$, where the $q_i$ are general quadratic forms of $S$. 
Let $I' =(f_1',\dots,f_{d+1}')S$, and set $J' = (y)+I'$. Note that the
leading terms of the $f_i'$ form a regular sequence.

By the semicontinuity of fiber dimension, 
the length of $S/J = R/I$ is at least the length of $S/J'$, so it
suffices to show that the latter is at least
${{d+1}\choose{\lceil \frac{d}{2} \rceil}}$.  

Since the leading terms of the $f_i'$ form a regular sequence,
  the $f_i$ are a standard basis in the sense of Grauert and
Hironaka. It follows that the Hilbert function of $S/I'$ is the same
as that of a complete intersection of quadrics,
$$
 \length((I'+\mathfrak{p}^m)/(I'+ \mathfrak{p}^{m+1}))={{d+1}\choose{m}}.
$$
Further, 
the length of $S/J' = S/((y)+I')$ 
is at least the length of the cokernel of the map given by multiplication 
$$ y: S/(I', \mathfrak{p}^{m}) \to S/(I', {\mathfrak p}^{m+1})$$ 
for any $m$. Thus 
$$
\length S/J' \geq \length S/(I', \mathfrak{p}^{m+1}) - \length S/(I', \mathfrak{p}^m) = 
\length((I', \mathfrak{p}^m)/(I', \mathfrak{p}^{m+1})).
$$
If we let $m = \lceil d/2 \rceil$, we get $\length S/J \geq {{d+1}\choose{\lceil d/2 \rceil}}$
as required.
\end{proof}

Applying the Stirling approximation for factorials we deduce an assymptotic formula:

\begin{corollary}
If $X$ is a smooth projective variety of dimension $n$ embedded by a sufficiently positive line bundle,
then a general projection $X$ into $\P^{n+1}$ will have some fibers whose 
degrees are of the order of $\sqrt{\frac{2}{\pi}} \; \frac{2^{\sqrt{n}}}{n^{1/4}}$. \qed
\end{corollary}

\section{Proof of The Main Theorem}

\begin{proof}[Proof of Theorem \ref{main1}] 
Let $\Sigma \subset \P^r$ be a general linear subvariety of 
codimension $n+c+1$ in $\P^r$, and let $\pi_{\Sigma}: X \to \P^{n+c}$ be projection from 
$\Sigma$, so that in particular
 $\pi_{\Sigma}$ is generically injective and is a finite map.
Thus if $\Lambda$ is 
any codimension $n+c$ linear subvariety of $\P^r$ that contains $\Sigma$, the intersection of $\Lambda$ and $X$ is a scheme of dimension zero. 
Fix a (general) $\Sigma$, and a $\Lambda\supset \Sigma$ that makes
the degree of $X \cap \Lambda$ maximal.

Consider the natural surjections of sheaves
$$
N_{\Lambda/\PP^r} = \Hom(\cI_{\Lambda}, \cO_\Lambda) \twoheadrightarrow
 \Hom(\cI_{\Lambda}, \cO_Z) \twoheadrightarrow Q(X, \Lambda),
$$
and let $F$ be the kernel of the composite map 
$\gamma: N_{\Lambda/\PP^r} \to Q(X, \Lambda).$
Identifying the local sections of $N_{\Lambda/\PP^r} $ with the 
embedded deformations of $\Lambda$ in $\PP^r$ (the tangent
space to the Grassmannian), we see from the definition of $Q(X,\Lambda)$
that the local sections of $F$ are the deformations of $\Lambda$
that induce flat deformations of $Z$. Since
$Q(X, \Lambda)$ is supported on $Z\subset X,$ which is disjoint from
$\Sigma$, we have
$F|_\Sigma=N_\Lambda|_\Sigma$. (See Theorem \ref{Q and deformations} for another
description of $F$.)
 
 Because $\Sigma$ is a hyperplane section of $\Lambda$
 there is an exact sequence
 $$
 0\to N_{\Lambda/\PP^r}(-1) \to N_{\Lambda/\PP^r} \to N_{\Lambda/\PP^r}|_\Sigma \to 0.
 $$
Putting this together with the exact sequence
 $$
 0\to F\to N_{\Lambda/\PP^r} \to Q(X, \Lambda)\to 0,
 $$
 coming from the definition of $F$,
  and taking global sections, we get a commutative diagram
  with exact row and column
$$\leqno{(*)}\qquad\qquad
\begin{diagram}[small]
&&H^0(N_\Lambda(-1))\cr
&&\dTo &\rdTo^{\beta} \cr
H^0(F)&\rTo&H^0(N_\Lambda)& \rTo_{\kern -30pt \gamma} &H^0(Q)= Q\cr
&\rdTo^\alpha&\dTo\cr
&&H^0(N_{\Lambda/\PP^r} |_\Sigma)= H^0(F|_\Sigma)
\end{diagram}
$$
where $\alpha, \beta$ and $\gamma$ are
defined by applying $H^0$ to the evident maps, and
$Q=Q(X, \Lambda)$.

There is an exact sequence of normal bundles 
$$
\xymatrix{
0 \ar[r] & N_{\Sigma/\Lambda} \ar[r]  & N_{\Sigma/\P^r} \ar[r]^{\psi} &  N_{\Lambda/\P^r}|_{\Sigma} \ar[r] & 0.
}
$$
Since $N_{\Sigma/\Lambda}= \cO_{\Sigma}(1)$, 
every global section of $N_{\Lambda/\P^r}|_{\Sigma}$
lifts to a global section of $N_{\Sigma/\P^r}$, that is, a deformation
of $\Sigma$.
The length of $X \cap \Lambda$ is semicontinuous as we move
$\Lambda$, and we have supposed that it is maximal among
those lengths attained by a $\Lambda$ containing a
general plane $\Sigma$; so if we move $\Sigma$ in a flat family $\Sigma_t\subset \PP^r$,
then (for small, or first-order $t$)
 there is a flat deformation $\Lambda_t\subset \PP^r$ of $\Lambda$
such that $\Sigma_t\subset \Lambda_t$
and the degree of $Z=X\cap \Lambda_t$ is 
constant---that is,  the family $Z_t := X \cap \Lambda_t$
is flat. Thus any first order deformation of $\Sigma$ can be lifted
to a first order deformation of $\Lambda$ fixing the length of 
$X \cap \Lambda$.

Using the identification
of sections of $F$ with deformations of $\Lambda$
fixing the length of $Z$, we see that
$$
H^0(F) \rTo^\alpha H^0(F|_{\Sigma})
= H^0(N_{\Lambda/\P^r}|_{\Sigma})
$$ 
contains the image of $H^0(N_{\Sigma/\P^r})$
under $\psi$. Since $H^1(N_{\Sigma/\Lambda})=H^1(\cO_\Sigma(-1))=0$,
$\psi$ is surjective on global sections, and thus $\alpha$ is surjective.

A diagram chase using the surjectivity of $\alpha$ and the exactness of the row and column in diagram (*) shows that 
the image of $\beta$ is equal to the image of $\gamma$. 
Moreover, $N_\Lambda(-1)\cong \O_{\Lambda}^{n+c}$ is 
generated by global sections.Thus we can apply
the following result with $A=N_\Lambda(-1)$ and $B=Q(X, \Lambda)$.

\begin{proposition}\label{dim0} Suppose that $\delta:A\to B$ is an epimorphism of 
coherent sheaves on $\PP^r$,
and suppose that $A$ is generated by global sections.
If $\delta(H^0(A)) \subset H^0(B)$ has the same dimension
as  $\delta(H^0(A(1))) \subset H^0(B(1))$, then $\dim B = 0$ and
$\delta(H^0(A(m)))=H^0(B(m))\cong H^0(B)$ for all $m\geq 0$.
\end{proposition}
\begin{proof}
We may harmlessly assume that the ground field is infinite, so we may
choose a linear form $x$ on $\PP^r$ that does not vanish on any associated
subvariety of $B$. It follows that multiplication by $x$ is a monomorphism
on global sections, so, under our hypothesis, 
$$
x\cdot \delta(H^0(A)) = H^0(\cO_{\PP^r}(1))\cdot \delta(H^0(A)) = \delta(H^0(A(1))).
$$
It follows that 
$x^m\cdot \delta(H^0(A)) = H^0(\cO_{\PP^r}(m))\cdot  \delta(H^0(A))$ for all $m\geq 0$. Since $\delta$ is
a surjection of sheaves and $A$ is globally generated, this space is equal to 
$H^0(B(m))$ for large $m$. It follows that the Hilbert polynomial of $B$ is constant,
so $B$ is zero-dimensional, and $H^0(B(m))=\delta(H^0(A(m)))$ for all $m\geq 0$ as
required.
\end{proof}

Returning to the proof of Theorem \ref{main1} we
apply Proposition \ref{dim0}
and deduce that $\beta$ is surjective.
Thus 
$$
\length Q(X, \Lambda)\leq \dim_k H^0(N_\Lambda(-1)) = n+c,
$$
 as required.
\end{proof}

\section{Bounds on the invariant $q$}

We begin by explaining why
the ``expected'' value of $q(X,Y)$ is $\deg(X\cap Y)$. We write $\Hilb_{X}$ for the Hilbert scheme of $X$,
and
$T_{\Hilb_{X,[Z]}}$ for its tangent space
at the point corresponding to the scheme $Z$. If $[Z]$ is a smooth
point of the Hilbert scheme, and if $Z$ deforms in $X$ to a reduced
set of smooth points, then $\dim T_{\Hilb_{X,[Z]}} = (\deg Z)(\dim X)$, 
which we may thus consider to be the ``expected value''. We denote by
$T^1(Z)$ the Zariski tangent space to the deformation space of
$Z$. We denote by $\Der\ \cO_Z$ the module of $k$-linear derivations from $\cO_Z$ to itself.

\begin{theorem}\label{deformation connection}
Suppose that $X,Y\subset P$ are $k$-schemes of finite type with $X$ smooth and $Y$ locally a complete intersection
in $P$. If $Z:=X\cap Y$ is finite, then
$$\aligned
\dim_k Q(X,Y) &= (\deg Z)(\codim Y-\dim X)-(\dim_k T_{\Hilb_{X,[Z]}}-(\deg Z)(\dim X))\cr
&=(\deg Z)(\codim Y-\dim X)-\dim_k T^1(Z)+\dim_k \Der\ \cO_Z.
 \endaligned
 $$
\end{theorem}

\begin{proof}
To compute the dimension of $Q(X,Y)$, we note first that since $I_Y:=I_{Y/P}$ is locally
a complete intersection
the module $I_Y/I_Y^2$ is locally free over $\cO_Y$ of rank equal to the codimension of $Y$. It follows that
$\Hom(I_Y/I_Y^2, \cO_Z) \cong (\cO_Z)^{\codim Y}$. Using this and the definition of $Q$ we get
$$
\dim_k Q(X,Y) = (\deg Z)(\codim Y) - \dim_k \Hom((I_X+I_Y)/(I_X+I_Y^2), \cO_Z).
$$
On the other hand, we may identify $\Hom((I_X+I_Y)/(I_X+I_Y^2, \cO_Z)$ with the tangent space to the
functor of embedded deformations of $Z$ in $X$, that is, with the tangent
space to the Hilbert scheme $\Hilb_X$ at $[Z]$. There is an exact sequence
$$
\aligned
0\to \Hom(\Omega_{Z/k}, \cO_Z) \to &\Hom(\Omega_{X/k} \mid_Z, \cO_Z)
\to\cr
\Hom((I_X+I_Y)/&(I_X+I_Y^2), \cO_Z)
\to
T^1(Z)
\to 0
\endaligned
$$
(see Eisenbud \cite{eisenbud1}, Ex. 16.8.)
Since $X$ is smooth we have
$$
\dim_k\Hom(\Omega_{X/k} \mid_Z, \cO_Z) = (\deg Z)^{\dim X}
$$
and the desired formula follows.
\end{proof}

\begin{corollary}\label{expected length}
Let $X,Y\subset P$ be schemes of finite type with $X$ smooth and $Y$ locally a complete intersection
in $P$. Suppose that $Z=X\cap Y$ is finite, and consider the corresponding point
$[Z]$ in the Hilbert scheme $\Hilb_X$. If $[Z]$ is a smooth point and lies
in the closure of the locus of reduced subschemes, then
$$q(X,Y) = \frac{\dim_k Q(X,Y)}{\codim Y - \dim X} = \deg \cO_Z.$$
In particular, this is the case when $Z$ is a complete intersection.
\end{corollary}

\begin{proof}
The dimension of the closure of the locus of sets of $\deg Z$ reduced points has dimension
$(\deg Z)(\dim X)$.
\end{proof}

The case of complete intersections extends to that of schemes in the ``linkage class of a complete
intersection" described below.

\begin{remark} The tangent space to $Z$ in the Hilbert scheme of $X$ is often larger than
$\deg \cO_Z^{\dim X}$, for example when $[Z]$ is at a point where
the ``smoothing component'' meets another component---a result of
Iarrobino \cite{Iarrobino72} shows that there always are
such points when $\dim X >2$. But it can also be smaller.
The first such example was discovered by Iarrobino and Emsalem \cite{EI} 
(see
\cite{hilb} for an exposition): the Hilbert scheme of finite subschemes of
degree 8 in $\C^4$ contains a reduced component 
 whose generic point is the scheme defined by 7 general
quadrics. This component is isomorphic to the product of $\C^4$ and the Grassmanian $Gr(7, 10)$, which 
has dimension 25, whereas the locus of reduced 8-tuples of points has dimension
32. Such a point can appear as the intersection of a 4-plane $X$ with a scheme $Y$ of
codimension 7. In this case we have
$\dim_k Q(X,Y) = 31 >\deg Z(\codim Y-\dim X) = 24$.
\end{remark}

\subsection*{ Lower Bounds on the invariant $q$.} Since the computation
of $Q$ is local, it suffices to treat the local case. If $\cO$ is a regular
local ring, then ideals $J, J'\subset \cO$ (or the subschemes they define)
are said to be linked by a complete intersection $K$ if $J'=(K:J)$ and $J=(K:J')$,
where $(K:J)$ denotes the ideal $\{f\in \cO\mid fJ\subset K\}$.

The ideal $J$ is
\emph{in the linkage class of a complete intersection}, written \emph{licci},
 if it can be linked to a complete intersection in finitely many steps.
 See Peskine-Szpiro \cite{PS} for general information about this notion.
 
We write $\mu(Q)$ for the minimal number of generators of an $\cO$-module $Q$,
and it is obvious that $\length Q\geq \mu(Q)$.
Example \ref{quadric graph} suggests that the length of $Q(X,Y)$ can
be equal to $\mu(Q(X,Y))$ even in large cases.

\begin{theorem}\label{qlength}
Let $\cO$ be  an equicharacteristic 0 regular local ring of dimension $r$.
Suppose  $X\subset \Spec \cO$ is smooth and $Y\subset \Spec \cO$ is a complete intersection. Set $c=\codim Y-\dim X$ and $Z=X\cap Y$. Suppose that
$c \geq 1$, and that  $\dim Z= 0$.
\begin{enumerate}
\item \label{qlength 1} If $I_Z$ is licci then 
$$ 
q(X,Y) =  \deg (X\cap Y).
$$ 
\item  \label{qlength 2}If $I_Z$ is not licci, then 
$$
q(X,Y)\geq \frac{1}{c}\mu(Q(X,Y))\geq  \max(1+\frac{3}{c}, \frac{5}{c}).
$$
\end{enumerate}
In particular, $q(X,Y)\geq 1$.
\end{theorem}

Here are some facts about the licci property that will be important to us:
 
 \begin{proposition}\label{licci lemma} Suppose that $A$ is a local Gorenstein ring
 with maximal ideal $\gm$, and $L\subset A$
 is an ideal of finite projective dimension such that $A/L$ is Cohen-Macaulay.
 \begin{enumerate}
\item \label{licci lemma 1} If $f_1,\dots f_m\in L$ is a regular sequence, and 
 $L/(f_1,\dots, f_m)$ is licci, then $L$ is licci.
\item\label{licci lemma 2} If $\codim L\leq 2$ then $L$ is licci.
\item\label{licci lemma 3} If $\mu(L)\leq 4$ then $L$ is licci.
 \item\label{licci lemma 4} Suppose $A$ is regular, and $A/L$ is of finite length.
Let  $T:=\gm/(\gm^2+L)$ be the Zariski tangent space  of $A/L$. 
If $\dim T \leq 2$, or if $L$ is generated by $\codim L+1$ elements and 
 $\dim T\leq 3$,  then $L$ is licci.
 \end{enumerate}
 \end{proposition}
 \begin{proof}[Proof Sketch]
The first assertion is immediate from the definition.
The second was proved by Ap\'ery and Gaeta \cite{Gaeta}, 
 and is reproved in modern language in Peskine-Szpiro \cite{PS}.
 
 When $L$ is generated by $\leq 2$ elements then it has codimension
 $\leq 2$ and is thus covered by the Theorem of Ap\'ery and Gaeta. If $\mu(L) = \codim L$,
 then $L$ is itself a complete intersection. The only remaining case
 with $\mu(L)\leq 4$ is the case of a 4-generator ideal of codimension 3.
 It follows from the paper of Peskine-Szpiro that $L$ is then linked to
 an ideal $L'$ of finite projective dimension that has a symmetric resolution,
 and the main theorem of Buchsbaum-Eisenbud \cite{Gor} shows
 that $L'$ is generated by the $2n\times 2n$ pfaffians of a $2n+1\times 2n+1$
 matrix. Ideas
 similar to those of Gaeta show that such ideals are licci; see also
 Watanabe \cite{Watanabe} that implicitly contains
 a different (and prior) proof of
  the slightly restricted case where $A$ is regular. 
 This proves part (3).
 
 Part (4) follows from the previous parts: if $\dim T\leq 2$, then
 there is a regular sequence $g_1,\dots g_s$ in $L$, with
 $s=\dim A -\dim T$, such that
 $A/(g_2,\dots, g_s)$ is again regular. If $\dim T=2$ we may apply
parts (1) and (2) to conclude that $L$ is licci. If $\dim T= 3$
then $\mu(L/(g_1,\dots,g_s))= \mu(L) - s= \codim L+1 - s =
\dim A+1-(\dim A-3) = 4$ so we may apply parts (1) and (3)
to conclude that $L$ is licci.
 \end{proof}

\begin{proof}[Proof of Theorem \ref{qlength}]
First suppose that $I_Z$ is licci. We may harmlessly complete the ring $\cO$
and thus can apply the result of
Buchweitz \cite{buchweitz}, Theorem 6.4.4 (p.~235), which shows that a licci scheme
represents a smooth point on its Hilbert scheme
(this result is proven in the analytic category). 
Ulrich \cite{ulrich} Theorem 2.1 implies that a licci scheme of
dimension at most 3 is smoothable. Part (1)
thus follows from Theorem \ref{deformation connection}.

To prove part (2), set
$\mu = \mu(Q(X,Y))$ and $n = \dim Y$. Let $J=(I_X+I_Y)/I_X$ be the image of $I_Z$ in 
$\cO_X=\cO/I_X$, and consider the  
defining exact sequence 
$$ 
0 \to \Hom(J/J^2, \cO_Z) \to \Hom(I_Y/I_Y^2, \cO_Z) \to Q(X,Y)  \to 0.
$$
Since 
$\Hom (I_Y/I_Y^2, \cO_Z)$ is a free $\cO_Z$-module of rank $r-n$, 
$\Hom (J/J^2, \cO_Z)$ 
must have an $\cO_Z$-free summand of rank $m:=r-n-\mu$. Since $\cO_Z$ is artinian, this implies that 
$J/J^2$ has also an $\cO_Z$-free summand of rank $m$. 

 We 
may write
this free summand in the form $J/K$ for some ideal $K$ 
of $\cO_X$ such that $J\supset K\supset J^2$. 
Since $\cO_X$ is regular, $J$ has finite projective dimension. The proof of 
Theorem 1.1 of Vasconcelos \cite{vas} shows that there is a regular sequence 
$f_1, \dots, f_{m}$ in $J$ such that 
 $$
 J = (f_1, \dots, f_{m}) + K
 $$
 and such that $J/(f_1, \dots, f_{m})$ is an ideal of 
 finite projective dimension in the ring  $\cO_X/(f_1, \dots, f_{m})$. 
 By Proposition \ref{licci lemma}.\ref{licci lemma 1} it suffices to show that
 when $\mu\leq c+2$ or $\mu \leq 4$ the ideal $J/(f_1,\dots f_m)$ is
 licci.
 
 If $\mu \leq c+2$ then 
 $$
 \dim \cO_X/(f_1,\dots, f_m) = r-n-c-m = (r-n-c) - (r-n-\mu)=\mu-c\leq 2
 $$
 so  Proposition \ref{licci lemma}.\ref{licci lemma 2} shows that $J$ is licci.
 On the other hand, if $\mu\leq 4$, then 
 $$
 \mu(J/(f_1,\dots,f_m)) = \mu(J)-m\leq \mu(I_Y) - m = r-n-m = (r-n)-(r-n-\mu) \leq 4
 $$
 so  Proposition \ref{licci lemma}.\ref{licci lemma 2} shows that $J$ is licci.
\end{proof}

If $\pi: X\to P$ is a map of smooth varieties, and $x\in X$ is a point, then the tangential corank of $\pi$ at $x$ is the 
dimension of the Zariski tangent space of the fiber to $\pi$ through $x$. Theorem \ref{qlength}
allows us to analyze the invariant $q$ at points of small tangential corank.

\begin{corollary}\label{corank2} 
Let $x\in X\subset \PP^r$ be a point on a smooth projective variety of dimension $n$, and let
$\pi:X \to \PP^{n+c}$ be a linear projection from a center that does not meet $X$, so that
the fiber $Z$ of $\pi$ through $x$ is the intersection of $X$ with a linear subspace $\Lambda$
of dimension $r-n-c$.
If the tangential corank of $\pi$ at $x$ is $\leq 2$, or
$c=1$ and the tangential corank of $\pi$ at $x$ is $\leq 3$,
then $q(X,\Lambda) = \deg (X\cap \Lambda)$.
\end{corollary}

\begin{proof} From Proposition \ref{licci lemma}.\ref{licci lemma 4} we see that the fiber
is licci at $x$, and the assertion then follows from Theorem \ref{qlength}.
\end{proof}

\begin{example} \label{quadric graph}
Let $A=k[a_1,\dots,a_n, x_1,\dots,x_{n+1}]$ and let 
$$
I=(x_1-f_1(a), \dots, x_{n+1}-f_{n+1}(a)); \quad X= V(I)\subset \Spec A
$$
 where the $f_i(a)$ are generic
quadrics in the variables $a_i$. Let $\Lambda=V(x_1,\dots,x_{n+1})$. The 
variety $X$ is smooth (it is the graph of the map $f := (f_i)$) and meets the plane
$\Lambda$, its tangent plane, in a scheme $Z$ supported at the origin. 

For $2\leq n\leq 8$ and random examples over a large
finite field, the values of $\deg (Z)$, $q(\Lambda,X)$,
which is equal to $q(X, \Lambda)$ by Theorem \ref{di}, and 
the minimal number of generators of the module $Q(\Lambda,X)$, 
written $\mu(Q(\Lambda,X))$,
are given in the following table, computed with
Macaulay2 \cite{M2}. Note that for even $n$ we have
$
\mu(Q(\Lambda,X)) = q(\Lambda,X) = \codim I.
$
It follows immediately that in these cases
 $$Q(\Lambda,X)=I/(a_1,\dots,a_n, x_1,\dots, x_{n+1})I,
 $$
 a vector space concentrated in degree 2.
 \medbreak
 
\centerline{ 
\begin{tabular}{l*{6}{c}r}
$n$ & $\deg Z$ & $q(\Lambda,X)$&$\mu(Q(\Lambda,X))$\\
\hline
2 & 3 & 3&3 \\
3 & 6 & 6&3 \\
4 & 10 & 5&5\\
5 & 20 & 20&6\\
6 & 35 & 7&7\\
7 & 70 & 57&8\\
8 & 126 & 9 &9
\end{tabular}
}
\medskip

In the notation
of Theorem \ref{qlength} we have $c=1$. In these examples
we have $\deg(\Lambda\cap X) = q(\Lambda,X)$ if and only if $n=2,3$,
and these are exactly the cases where $\Lambda\cap X$ is licci. 

The case $n=4$ shows that the bound $q(\Lambda,X)\geq 5$
 in part (2) of Theorem \ref{qlength} can be sharp.
For an example where the other option is sharp consider the smallest
non-licci scheme, which is $Z=\Spec k[x,y,z]/(x,y,z)^2$. This scheme can be written
as the intersection of a smooth complete intersection $X$ of 6 general quadrics in 
some $\PP^r \ (r\geq 7)$ with
a 3-plane $\Lambda$, and thus $c=3$. It is easy to compute the dimension of the tangent space to the Hilbert scheme
of $\Lambda$ at $[Z]$ has dimension 18, so by Theorem \ref{deformation connection}---or direct computation---
$$
q(\Lambda,X) = \mu(Q(\Lambda,X))=6=c+3.
$$
\end{example}
\medbreak

\section{secant lines}

Let $X \subset \P^r$ be a smooth subvariety of dimension $n$, and denote by 
$S_l$, $l\geq 2$, the subvariety of  $\P^r$ swept out by all the $l$-secant lines of $X$. Let 
$c = \lfloor\frac{n}{l-1} \rfloor +1$, and 
let  $\Sigma$ be a linear subspace of codimension $n+c+1$ in $\P^r$. If 
$S_l$ intersects $\Sigma$ at a point $q$, then there is an $l$-secant line $\Lambda_1$
 of $X$ which 
passes through $q$, and this line together with $\Sigma$ span a linear subspace $\Lambda$ of codimension 
$n+c$ in $\P^r$ that intersects $X$ in a scheme of degree $\geq l > \frac{n+c}{c}$. 
If we knew that $q(X,\Lambda)$ was bounded below by
$\deg(X\cap \Lambda_1)$, then Theorem \ref{main1} 
would show that a general such $\Sigma$ does not intersect $S_l$, and so 
$\dim S_l \leq n+c$. Though we do not know how to prove such a comparison
theorem for $q(X,\Lambda)$, 
Theorem \ref{main3} shows that this upper bound 
on the dimension of $S_l$ is satisfied. The examples in 
Section \ref{examples section} show that it is sharp. 

\begin{proof}[Proof of Theorem \ref{main3}]
Let $G_{l}$ be the subvariety of the Grassmannian of lines in $\P^r$ parametrizing 
the lines whose intersection with $X$ has degree $\geq l$, and 
let $G^0$ be an irreducible component of $G_{l}$. We must show that the lines 
parametrized by $G^0$ sweep out a subvariety of dimension $\leq n+c$ in $\P^r$. 
 
Let $[\Lambda]$ be a general point in $G^0$. We may assume
$\Lambda\not\subset X$, since otherwise this conclusion is obvious. Set
$Z=X\cap\Lambda$, and let $F=\ker( N_{\Lambda/\PP^r} \to Q(X,\Lambda))$. Since $F$ is a torsion free sheaf of rank $r-1$ on $\P^1$, it splits as 
$F = \O_{\Lambda}(a_1) \oplus 
\dots \oplus \O_{\Lambda}(a_{r-1}), 
a_1 \geq a_2\geq \dots \geq a_{r-1}.$
We see from the definition that the Euler characteristic of $F$ is
$r-1+\sum_i a_i = 2(r-1) - \length(Q(X,\Lambda))$.
Since any subscheme of $\P^1$ is a local complete intersection, 
part (1) of 
Theorem \ref{qlength} yields
\begin{equation}\label{ineq1}
\sum_i a_i =r-1 - \length(Q(X,\Lambda)) = r-1-(r-n-1)l.
\end{equation}
We next show that $a_i \geq  -l+1$ for every $i$ by showing
that $H^1(F(l-2))=0$. 

Consider the short exact sequence 
$$
0 \to F(l-2) \to N_{\Lambda/\P^r}(l-2) \to Q(X, \Lambda) \to 0.
$$
Since $H^1(N_{\Lambda/\P^r}(l-2))=0$, to prove $H^1(F(l-2))=0$ 
 we need to show that the map  
$H^0(N_{\Lambda/\P^r}(l-2)) \to H^0(Q)$ is surjective.
This map factors through 
$$
H^0(N_{\Lambda/\P^r}(l-2)) \to H^0(N_{\Lambda/\P^r}(l-2)|_Z) \to H^0(Q). 
$$
The first map is surjective since the Castelnuovo-Mumford regularity of a finite scheme 
is bounded by the degree of the scheme, and the second map is surjective since 
$ N_{\Lambda/\P^r}|_Z \to Q(X, \Lambda)$ is surjective and $Z$ is zero-dimensional. This shows that $a_i \geq -l+1$ for every $i$.

Let now $k$ be the largest index such that $a_k \geq 0$, then 
\begin{equation}\label{ineq2}
  \sum_i a_i \geq (r-1-k)(-l+1).
\end{equation}
Combining (\ref{ineq1}) and 
(\ref{ineq2}), we get 
$ r-1-(r-n-1)l \geq (r-1-k)(-l+1)$, and so $k \leq \frac{nl}{l-1}$.

Let $\mathcal I \subset \P^r \times G^0$ be the incidence 
correspondence, and denote by $p_1$ and $p_2$ the two projections from 
$\mathcal I$ to $\P^r$ and $G^0$. We get a commutative diagram 
$$
\xymatrix{
  T_{\mathcal I, ([\Lambda], p)}
  \ar[r]^-{dp_2} \ar[d]^{dp_1} & 
  \;\;\;\; T_{G^0, [\Lambda]} = H^0(F) \ar[d] \\
  T_{\P^r, p}  \ar[r] & N_{\Lambda/\P^r}|_p = F|_p.
}
$$
 If $([\Lambda], p)$ is a general point of 
$\mathcal I$, then since $k \leq \frac{nl}{l-1}$, the image of the restriction map $H^0(F) \to 
F|_p$ has dimension at most $\frac{nl}{l-1}$. Therefore,  the dimension of the image of $dp_1$, which is equal to the dimension
of the subvariety swept out by all the lines parametrized by $G^0$, is at most 
$\frac{nl}{l-1}+1$.  
\end{proof}

A similar argument proves an analogous result 
for the $2$-dimensional linear subvarieties intersecting $X$ in a scheme of dimension zero:

\begin{theorem}\label{plane}
If $S_{l,t}$ is the closure of the subvariety of $\P^r$ swept out by the $2$-planes $\Lambda$ such that the 
intersection of $\Lambda$ and $X$ is a scheme of degree at least $l$ and regularity at most $t$, 
then 
  $$\dim S_{l,t} \leq \frac{  {{t+1}\choose{2}} (r-2) -l(r-2-n)} {  {{t}\choose{2}}  } + 2.$$
\end{theorem}

\begin{proof}
Let $G^0$ be an irreducible component of the 
space of $2$-planes in $\P^r$ which intersects $X$ in a scheme of degree at least $l$ and regularity at most $t$, and 
let $[\Lambda]$ be a general point in $G^0$. Consider the exact sequence 
$$ 
0 \to F \to N_{\Lambda/\P^r} \to Q(X,\Lambda) \to 0.
$$
Since $Z = \Lambda \cap X$ is $t$-regular, an argument parallel to the one given
in the proof of Theorem \ref{main3} shows that the map $H^0(N_{\Lambda/\P^r}(t-2)) \to H^0(Q(X,\Lambda))$ is surjective.
Thus,
$$
h^0(F(t-2)) = {{t+1}\choose{2}} (r-2) -l(r-2-n).
$$
On the other hand, if the planes parametrized by $G^0$ cover a subvareity of dimension at least $k$, then for a general point $p \in \Lambda$, the image of the restriction map 
$H^0(F) \to F|_p$ is at least $(k-2)$-dimensional, so 
$$
h^0(F(t-2)) \geq H^0(\O(t-2)^{k-2}) = {{t}\choose{2}} (k -2),
$$
 and we get the desired bound.
\end{proof}

\section{Examples}\label{examples section}
Here are
examples showing that the bounds in Theorem \ref{main3} are sharp,
and that Conjecture  \ref{fiber-reg} has the best possible bound on the regularity of the fibers.

\begin{example} \label{final examples}\label{codim 2 complete intersections}

Let $r = \lfloor \frac{nl}{l-1}+1 \rfloor$, 
and let $X$ be the complete intersection of 
$r-n$ general hypersurfaces $Y_1, \dots, Y_{r-n}$ of degree $l$ in 
$\P^r$. We show that $\P^r$ is swept out by the $l$-secant lines of 
$X$ and so 
$$ \dim S_{l} = r = \lfloor \frac{nl}{l-1}+1 \rfloor .$$
Consider the intersection of any $r-n-1$ of the $Y_i$, say 
$Y = Y_1 \cap \dots \cap Y_{r-n-1}$. Under these circumstances,
it is known that $Y$ is covered by
lines; we give a proof for the reader's convenience:

Let $p$ be a point in $Y$, let $V$ be the cone of lines in $Y$ through $p$, and let
$V_i$ be the cone of lines in $Y_i$ through $p$. 
We can compute the equations of $V_i$ as follows. 
We may assume that $p = (1:0:\dots:0)$ and write the equation of $Y_i$ in the form 
$$
x_0^{l-1} F^i_{1}+ \dots + x_0 F^i_{l-1} + F^i_l,
$$
where $F^i_{d}$ is a form of degree $d$ in $x_1, \dots, x_n$.
Consider the line $W$ through $p$ and another point, which we may take
to be $p' = (0:1:0:\dots:0).$ Substituting the parametrization
$(1:t:0:\dots:0)$ of $W$ into the equation of $Y_i$, we see that $ \{F_1^i(p')= \dots 
=F_{l}^i(p') = 0\}$ if and only $W$ lies in $Y_i$. 
Thus $\codim V_i \leq l$. Since 
$$ 
\codim V \leq \sum_{1 \leq i \leq r-n-1} \codim V_i = (r-n-1) l \leq r-1,
$$
the cone $V$ is at least 1-dimensional; that is, there is at least one line
through $p$ contained in $Y$. 

Any line in $Y$ 
intersects $X$ in $l$ points, so any point of $Y$ 
is contained in $S_l$. Since any point of $\P^r$ is contained in the 
intersection of $r-n-1$ 
independent hypersurfaces in the linear system spanned by 
$Y_1, \dots, Y_{r-n}$, there is an $l$-secant line of  $X$ passing through every point of $\P^r$.

In particular, in the case $l = n+1, r=n+2$, this argument shows that if $X^n\subset \PP^{n+2}$
is a complete intersection of two surfaces of degree $n+1$, 
then any projection of $X$ to $\PP^{n+1}$ has fibers of length $n+1$.
This consequence is greatly generalized by the argument of Lazarsfeld given in 
Proposition \ref{lazarsfeld argument} below.
\end{example}

\subsection*{A nondegenerate surface in $\PP^5$ with many trisecants.}
\begin{example}\label{enriques}

The Fano model of the classical Reye Congruence is a non-degenerate 
Enriques surface in $\P^5$ whose $3$-secant lines sweep out a $4$-dimensional subvariety in $\P^5$
(see Conte and Verra \cite{conte+verra} Propositions 3.10 and 3.14, and Cossec \cite{cossec}, Section 3.3.)
It can be described as follows.
Let $A=[f_{ij}]_{1 \leq i,j \leq 4}$ be a symmetric $4 \times 4$ matrix whose entries are general linear forms in $\P^5$, and let $S$ be the subvariety of 
$\P^5$ defined by all the $3 \times 3$ minors of $A$. Then $S$ is a smooth nondegenerate surface in $\P^5$. We show that 
the $3$-secant lines of $X$ sweep out the degree 4 hypersurface in $\P^5$ defined by the determinant of $A$.  
Since $A$ is general, for a general vector $V=[g_1,g_2,g_3,g_4]$ in the row space of $A$, the $g_i$ are independent linear forms and their 
intersection defines a line $l_V$ in $\P^5$. Then $l_V$ intersects $X$ in a scheme of degree $3$: without loss of generality, we can assume 
that $V$  forms the first row of $A$; since $A$ is symmetric, the first column of $A$ vanishes on $l_V$ too. So the intersection of $l_V$ and $X$ is 
the same as the intersection of $l_V$ and the $3 \times 3$ minor $[f_{ij}]_{2\leq i,j \leq 4}$ which is a scheme of degree 3.  If $p$ is a general point in 
$\{ \det A = 0 \}$, then  there is a general vector in  the row space of $A$ which vanishes on $p$, and so $p$ is contained in a 3-secant line of $X$. 
\end{example}

This is the only example we know of a non-degenerate smooth $n$-fold in $\PP^{2n+1}$ with $n\geq 2$ such that
some fibers of a general projection to $\PP^{n+1}$ have regularity $n+1$---that is, the fiber consists of
points contained in a line. Are there other such examples? 

\subsection*{Codimension 2}
If $X\subset \PP^{n+2}$, then every fiber of a projection to $\PP^{n+1}$ is contained in a line,
so the regularity of each fiber is equal to its degree. Thus one can check the
degrees of fibers by checking their regularity. The following result was shown us by Rob
Lazarsfeld; with his generous permission we include a proof along the lines he suggested.

\begin{proposition} [Lazarsfeld] \label{lazarsfeld argument}Suppose that $X \subset \P^{n+2}$ is 
arithmetically Cohen-Macaulay
of dimension $n$. If $X$ does not lie on any hypersurface
of degree $< n+1$,  then any projection of $X$ to $\P^{n+1}$ from a point off $X$
must have fibers of degree (and regularity) at least $n+1$; that is,  
the closure of the union of $n+1$-secant lines to $X$ fills $\PP^r$. 
\end{proposition}

We note that Zak's famous theorem on linear normality (\cite{Zak}, Chapter 2) can
be phrased in a similar way: it says that if $X^n$ is a smooth subvariety of 
codimension $\leq 1+ \lceil \frac{n}{2}\rceil$
in $\PP^r$, not contained
in a hypersurface of degree $<2$, then the 
 the closure of the union of 2-secant lines to $X$  fills $\PP^r$.
Is there a nice statement about $k$-secant lines, $2<k< n+1$, that interpolates
between these two results?

According to Hartshorne's Conjecture (\cite{Hartshorne-Bulletin}, Introduction), smooth 
projective varieties of codimension 2 and dimension $>6$
are complete intersections, so Proposition \ref{lazarsfeld argument} may include
all codimension 2 varieties of dimension $>4$. For dimension 2,
we examined 48 examples
of surfaces in $\PP^4$ catalogued by Decker, Ein, Schreyer \cite{DES} and Popescu \cite{Popescu}
to see whether their trisecants fill $\PP^4$. These examples were produced
using code originally written in the program Macaulay Classic of Bayer and Stillman,
and translated into Singular by Oleksandr Motsak. We used the program Magma to
make the computations using algorithms based on the paper of Eisenbud-Harris \cite{eisenbud-harris} and
unpublished work of Eisenbud-Ulrich. Of the 48 examples, 45 lie on no quadrics. Of these 45,
there is just one whose trisecants do not fill $\PP^4$: the elliptic scroll of degree 5,
whose ideal is generated by 5 cubics.

We now turn to the proof of Proposition \ref{lazarsfeld argument}.
To compute things about a linear projection $\pi_\Sigma$ from a linear space $\Sigma\subset \PP^r$
to $\PP^{n+1}$, 
we resolve it by blowing up. The general setup is this:
Let $\Sigma$ be a plane of dimension $\lambda-1$, where 
$\lambda := r-n-1$. Let $\beta: B \to \PP^r$ be the blowup of $\Sigma$, and let
$E$ be the exceptional fiber. Then 
$$
B\cong \PP_{\PP^{n+1}}(\cE),
$$
 where
$$
\cE \cong \cO_{\PP^{n+1}} \oplus \cO_{\PP^{n+1}}(-1)^\lambda, \text{ and } \cO_{\PP_{\PP^{n+1}}(\cE)}(1) = \cO_B(E).
$$
With this notation the projection fits into the diagram
$$\begin{diagram}
 \PP_{\PP^{n+1}}(\cE)\cong B& \rTo^\beta& \PP^r \cr
\dTo^\alpha&\ldDashto_{\pi_\Sigma} \cr
\PP^{n+1} 
\end{diagram}
$$
where $\alpha: \PP_{\PP^{n+1}}(\cE) \to \PP^{n+1}$ is the structure map of the projective bundle. 
In these terms we can describe the functor $\alpha_*\beta^*$, and more generally the 
derived functors $\R^i\alpha_*\beta^* $, quite explicitly, at least for their
action on sums of line bundles. 

\begin{lemma}[\cite{Hartshorne} II, 17.11 and III, Ex. 8.4] \label{direct images} Let $\alpha,\beta$ and $\cE$ be as above.
\begin {enumerate}

\item \label{sections} There are canonical isomorphisms
$$
\alpha_*\beta^* (\cO_{\PP^r}(d)) \cong \cO_{\PP^{n+1}}(d) \otimes Sym_d(\cE) \cong \bigoplus_{j=0}^d \bigl( \cO_{\PP^{n+1}}(d-j) \otimes Sym_j(\O_{\PP^{n+1}}^\lambda)\bigr).
$$
These induce isomorphisms
$$
H^0(\alpha_*\beta^*( \cO_{\PP^r}(d))) \cong H^0(\cO_{\PP^{n+1}}(d) \otimes Sym_d(\cE)) \cong H^0(\cO_{\PP^r}(d)).
$$
In particular, $\alpha_*\beta^* (\cO_{\PP^r}(d)) = 0$ when $d<0$.

\item \label{Rtop} There are canonical isomorphisms
$$
\R^\lambda(\alpha_*)\beta^* (\cO_{\PP^r}(-d))
\cong \cO_{\PP^{n+1}}(-d) \otimes (Sym_{d-\lambda-1} (\cE))^*
\otimes \bigwedge^{\lambda+1}\cE^*.
$$
\end{enumerate}
\qed
\end{lemma} 

\begin{proof}[Proof of Proposition \ref{lazarsfeld argument}]
We adopt the notation of Lemma \ref{direct images}, with $\lambda = 1$,
and we write $\pi: X\to \PP^{n+1}$ for the projection restricted to $X$.
 To show that the
regularity of some fiber is at least $n+1$ it suffices to show that for some point
$y\in \PP^{n+1}$ we have
$$
H^1(\cI_{\pi^{-1}(y)}(n-1)) \neq 0.
$$
Since $\Sigma\cap X = \emptyset$ we have $\pi^{-1}(y) = \alpha^{-1}(y) \cap \beta^{-1}(X)$
and $\cI_{\beta^{-1}X} = \beta^*\cI_X$. Thus
it suffices to show that $H^1(\beta^*(\cI_X(n-1))|_{\alpha^{-1}(y))}) \neq 0$ for some $y$.
By the Theorem on Cohomology and Base-change (see for example Hartshorne \cite{Hartshorne}, Theorem III. 12.11])
if suffices finally to show that 
$\R^1\alpha_*(\beta^*(\cI_X(n-1))) \neq 0$.

Because we have assumed that $X\subset \PP^{n+2}$ is arithmetically Cohen-Macaulay, $\cI_X$
has a resolution 
$$
0\to F_1 \to F_0\to \cI_X\to 0
$$
where $F_0$ and $F_1$ are sums of line bundles.
The map $\beta$ is locally an isomorphism on $X$, and outside
$X$ the sheaf $\cI_X$ is locally free, so we may pull the resolution back by $\beta$ (after tensoring with
$\cO_{\PP^{n+2}}(n-1)$) to get a short exact sequence of sheaves on $B$ of the form
$$
0\to \beta^*(F_1(n-1)) \to  \beta^*(F_0(n-1))\to  \beta^*(\cI_X(n-1)) \to 0.
$$
Using the fact that the fibers of $\alpha$ are 1-dimensional, we get from this a right exact sequence
$$
\R^1\alpha_*(\beta^*(F_1(n-1))) \rTo^\phi  \R^1\alpha_*(\beta^*(F_0(n-1))) \to  \R^1\alpha_*(\beta^*(\cI_X(n-1))) \to 0.
$$
and it suffices to show that the map labeled $\phi$ is not surjective.

We may write
$$
F_1 = \oplus_{i=1}^{t-1} \cO_{\PP^{n+2}}(-e_i), \qquad F_0 = \oplus_{i=1}^{t} \cO_{\PP^{n+2}}(-d_i)
$$
for some integers $t, d_i, e_i$, and considering first Chern classes on $\PP^{n+2}$
we see that $\sum_{i=1}^{t-1} e_i = \sum_{i=1}^{t} d_i$. 
Applying Lemma \ref{direct images} we see that
$$
\R^1\alpha_*(\beta^*(F_1(n-1))) = \bigoplus_{i=1}^{t-1}
 \cO_{\PP^{n+1}}(-e_i+n-1) \otimes Sym_{e_i-n-1}(\cE)^*\otimes \wedge^2(\cE)^*,
$$
and similarly
$$
\R^1\alpha_*(\beta^*(F_0(n-1))) = \bigoplus_{i=1}^{t}
 \cO_{\PP^{n+1}}(-d_i+n-1) \otimes Sym_{d_i-n-1}(\cE)^*\otimes \wedge^2(\cE)^*.
$$
Because we have assumed that $X$ does not lie on a hypersurface of degree $<n+1$,
all the integers $d_i$ are $\geq n+1$, so all the summands are nonzero.
Thus 
\def\rank{{\rm rank}}
$$
\rank\  \R^1\alpha_*(\beta^*(F_0(n-1))) = 
\sum_{i=1}^{t} d_i -tn  =  \rank\  \R^1\alpha_*(\beta^*(F_1(n-1))) - n.
$$

Set $p := \sum_i d_i - tn$. Taking modules of twisted global sections, we may represent
$\phi$ by a map from a graded free module of rank $p+n$ over the homogeneous coordinate ring of
$\PP^{n+1}$ to another such module, of rank $p$.
It follows by Macaulay's Generalized Principal Ideal Theorem (see for example \cite{Eisenbud} Exercise 10.9)
that if such a map is not surjective, then its cokernel has codimension at most $n+1$. Thus
to prove that the cokernel is nonzero as a sheaf on $\PP^{n+1}$, it suffices to show that
the map $\phi$ does not induce a surjection on twisted global sections.
This is immediate from the direct sum decompositions given above:  after twisting by $\cO_{\PP^{n+1}}(1)$ we get
$$
h^0\R^1\alpha_*(\beta^*(F_0(n-1))) \otimes \cO_{\PP^{n+1}}(1) = t
$$
and 
 $$
h^0\R^1\alpha_*(\beta^*(F_1(n-1))) \otimes \cO_{\PP^{n+1}}(1) = t-1.
$$
\end{proof}

\section{Symmetry and Decomposition of $Q$} 

The sheaf $Q(X,Y)$ is defined in terms that can be understood
as the comparison of the deformations of $X$ in the ambient space to 
the deformations of $X\cap Y$ in $Y$. It turns out that under favorable
circumstances $Q(X,Y)=Q(Y,X)$, and this module is determined up
to a free summand by the scheme $X\cap Y$ itself.

\begin{theorem}\label{Q and deformations}\label{di}
Suppose that $X,Y\subset \PP^r$ are smooth subvarieties with ideal
sheaves $\cI_X, \cI_Y$, and set $Z=X\cap Y$.
Let $F\subset N_{Y/\PP^r} = Hom(\cI_Y, \cO_Y)$ be the subsheaf consisting of those local sections
that map $\cI_X\cap\cI_Y$ into $\cI_{Z/Y}\subset \cO_Y$. 
\begin{enumerate}
\item
If $\eta\in H^0 (N_{Y/\PP^r})$ is a flat first-order deformation of $Y$, then $\eta$ induces 
a flat first-order deformation of $X\cap Y$ if and only if $\eta\in H^0 (F)$.
\item
There is a natural short exact
sequence 
$$
0\to F\to N_{Y/\PP^r} \to Q(X,Y)\to 0.
$$
\item $Q(X,Y)=Q(Y,X)$. \label{symmetry in smooth case}
\end{enumerate}
\end{theorem}

\begin{proof} (1) We must show that $\eta\in H^0 (F)$ if and only
if there is a map $\psi: \cI_Z \to \cO_Z$ making the diagram
$$
\begin{diagram}[small]
\cI_Y&\rTo& \cI_Z &\lTo& \cI_X \cr
\dTo^\eta &&\dTo^\psi&&\dTo^0\cr
\cO_Y&\rTo& \cO_Z &\lTo&\cO_X \cr
\end{diagram} 
$$
commute. Writing $\overline \eta$ for the composition of $\eta$ with
the projection map $\cO_Y\to\cO_Z$, this commutativity is equivalent to the commutativity
of the diagram
$$
\begin{diagram}[small]
\cI_Y&\rTo& \cI_Z &\lTo& \cI_X \cr
&\rdTo_{\overline \eta} &\dTo^\psi&\ldTo_0\cr
&& \cO_Z && \cr
\end{diagram} 
$$
From the short exact sequence 
$$
0\to \cI_X\cap \cI_Y \to \cI_X\oplus \cI_Y \to \cI_Z\to 0
$$
we see that such a $\psi$ exists if and only if $\eta: \cI_Y\to \cO_Y$
induces the zero map $\cI_X\cap \cI_Y \to \cO_Y \to \cO_Z$, that is,
$\eta$ maps $\cI_X\cap \cI_Y$ into $\cI_{Z/Y}$, proving part (1).
\medskip

\noindent (2) It follows from the definition that 
$$
Q(X,Y)= \frac{\Hom(\cI_Y, \cO_Z)}{\{ f\mid  f(\cI_X\cap \cI_Y)=0\}},
$$
so the kernel of $N_{Y/\Pr} \to Q(X,Y)$ is equal to $F$.

 \noindent (3) Because $X$ is smooth, the derivation
$d: \cI_X \to \Omega_{\PP^r}$ induces a locally split injection
$d: \cI_X/\cI_X^2 \to \Omega_{\PP^r}|_X$.
It follows that the induced map of sheaves
$$
\Hom(\Omega_{\PP^r}, \cO_Z) \to \Hom(\cI_X, \cO_Z)
$$
is an epimorphism. Since $d: \cI_X \to \Omega_{\PP^r}$ factors
as the inclusion of $\cI_X\subset \cI_Z$ and the map
$d: \cI_Z \to \Omega_{\PP^r}$, we see that the restriction map
$$
 \Hom(\cI_Z, \cO_Z) \to \Hom(\cI_X, \cO_Z),
$$
is also an epimorphism. The same considerations hold for $X$
in place of $Y$.

Consider the commutative diagram
$$
\begin{diagram}[small]
&&g\in\Hom(\cI_Y, \cO_Z)&\rTo&Q(X,Y) = \frac{\Hom(\cI_Y, \cO_Z)}{\{ f\mid  f(\cI_X\cap \cI_Y)=0\}}
\cr
 &\ruTo&&\rdTo\cr
\tilde g\in  \Hom(\cI_Z, \cO_Z) &&&&\Hom(\cI_X\cap \cI_Y, \cO_Z)\cr
 &\rdTo&&\ruTo\cr
&&\overline g \in \Hom(\cI_X, \cO_Z)&\rTo&Q(Y,X) = \frac{\Hom(\cI_X, \cO_Z)}{\{ f\mid  f(\cI_X\cap \cI_Y)=0\}}
\end{diagram}
$$
where the diagonal maps are restriction homomorphisms. 
By the argument above, the two maps
coming from $\Hom(\cI_X, \cO_Z)$ are epimorphisms.
We define a map $\phi:  \Hom(\cI_Y, \cO_Z) \to Q(Y,X)$ as follows. 
Given a local section $g\in \Hom(\cI_Y, \cO_Z)$ then on a sufficiently small open set of $\PP^r$
we may lift $g$ back to a local section $\tilde g\in  \Hom(\cI_Z, \cO_Z)$.
Let $\overline g$ be the image  of $\tilde g$
in $\Hom(\cI_X, \cO_Z)$, and let $\phi(g)$ be the image of $\overline g$ in $Q(Y,X)$.

If $\tilde g'$ a different lifting, then $\tilde g-\tilde g'$
goes to zero in $\Hom(\cI_X\cap \cI_Y, \cO_Z)$, and thus $\phi(g)$ is well-defined.
It follows at once that $\phi$ is a homomorphism, and since the map
$\Hom(\cI_Z, \cO_Z)\to \Hom(\cI_X, \cO_Z)$ is an epimorphism, so is $\phi$.
Moreover, $\phi$ annihilates the maps $f$ such that $f(\cI_X\cap \cI_Y) =0$,
so $\phi$ induces an epimorphism $Q(X,Y)\to Q(Y,X)$. The inverse map is
constructed by a symmetrical procedure, proving part (3).

 \end{proof}

\begin{corollary}\label{h}
Suppose that $X,Y\subset \PP^r$ are smooth varieties that meet in 
a finite scheme $Z$ of degree $l$, and let $F$ be the kernel of the
surjection $N_Y\to Q(X,Y)$ defined above. 
Suppose that $S$ is a reduced algebraic subset of the Hilbert scheme
of subschemes of $\PP^r$ containing the point [Y] that corresponds to $Y$.
If the points of $S$ near $p$ correspond to
subschemes that meet $X$ 
in schemes of degree at least $l$, then the tangent space to $S$ at $p$ is
a subspace of $H^0 (F)$. In particular, the dimension of the 
tangent space to $S$ at $p$
 is at most $h^0(N_Y) - \dim_k Q(X,Y)$.
\end{corollary}

\begin{proof}
Let $\pi: \mathcal Y\to S$ be the restriction of the universal
family. By semicontinuity, the degree of the intersection $\pi^{-1}(q)\cap X$, for
$q\in S$ near $p$, is equal to $l$. Since $S$ is reduced, this guarantees
that the family of intersections 
$$
\begin{diagram}[small]
\mathcal Y\cap (X\times S)&\ \ \subset\ \ &\PP^r\times S\cr
 \dTo &\ldTo\cr
 S
 \end{diagram}
$$
is flat over a neighborhood of $p$. The tangent space to $S$ at $p$ thus 
consists of first-order deformations of $Y$ that induce flat deformations of $X\cap Y$,
and the first statement follows from part (1) of Theorem \ref{Q and deformations}. 
The dimension statement is then immediate from
the exact sequence in part (2) of the Theorem.
\end{proof}

To put part (\ref{symmetry in smooth case}) of 
Theorem \ref{Q and deformations} into context, we note that, under somewhat more general
circumstances, the structure of $Q(X,Y)$ depends mostly on the intersection
$Z=X\cap Y$.

\begin{proposition}\label{independence}
Let $Z$ be a finite scheme over $k$. There is
a module $\overline Q(Z)$ depending only on $Z$ such that if
$X,Y\subset P$ are $k$-schemes of finite type with $X$ smooth and $Y$ locally a complete intersection
in $P$ such that $Z=X\cap Y$,  then $Q(X,Y) \cong \overline Q(Z) \oplus \cO_Z^m$ for some $m$.
\end{proposition}

\begin{proof} From the definition we see that $Q(X,Y)$ is the direct sum of local 
contributions, so we may harmlessly suppose that $Z$ has only one closed point.
Choose a minimal surjection
 $\phi: A:=k[[x_1,\dots,x_n]]\to \cO_Z$ so that $n$ is the dimension of the Zariski
 tangent space of $\cO_Z$. Let $I_Z$ be the kernel of this map, and let
 $f: F\to I_Z$ be a minimal surjection from a free $\cO_Z$-module. Set 
 $$
 \overline Q(Z, \phi, f) = \coker\biggl( \Hom_P(I_Z/I_Z^2, \cO_Z) \rTo^{\Hom(f, \cO_Z)} \Hom_P(F, \cO_Z)\biggr).
 $$
 We will show that $\overline Q(Z):=\overline Q(Z,\phi, f)$ is independent of the choices 
 of the minimal surjections $\phi$ and $f$. Note that $\overline Q(Z,\phi, f)$ has no free summand.
 
First, if $f': F'\to I_Z$ is any surjection from a free $\cO_Z$-module, then we may
write $F'\cong F\oplus G$ in such a way that $f'$ becomes the map $(f,0)$
so $\overline Q(Z,\phi,f') = \overline Q(Z,\phi,f)\oplus (G\otimes_P \cO_Z)$. In particular,
this shows that $\overline Q(Z,\phi,f)$ is independent of the choice of $f$ so long as $f$ 
is minimal.
 
Next, if $A'\to \cO_Z$ is a surjection from a different power series ring (of any dimension) then
we may choose a third power series ring $A''$ surjecting onto both $A'$ and $A$. It thus
suffices to show that, if $\psi: A''\to A$ is a surjection
of power series rings, then
$\overline Q(Z, \phi\psi, f)=  \overline Q(Z, \phi, f) \oplus \cO_Z^m$,
where $m$ is the difference $d$ of dimensions between $A$ and $A''$. 
Let $(y_1,\dots,y_d)$ be the kernel of $\psi$. Lifting generators of $I_Z$ back
to $A''$ as power series independent of the $y_i$, we see that the
kernel $I''_Z$  of the surjection $A''\to \cO_Z$ may be written as $(y_1,\dots, y_d)+I'_Z$, where
$y_1,\dots, y_d$ are a regular sequence modulo $I'_Z$. It follows that
$I''_Z/{I''_Z}^2=I'_Z/{I'_Z}^2 \oplus \cO_Z^d \cong I_Z/I_Z^2\oplus  \cO_Z^d$.
This shows that $\overline Q(Z,\phi,f)$ is independent of the choices of $\phi$ and $f$.

If $X,Y\subset P$ are schemes, with $X$ smooth, $Y$ a complete
intersection, and $Z=X\cap Y$, then $\cO_Z$ is
a homomorphic image of the completion $\hat {\cO}_{X,Z}$ of the local ring  of $X$ at the closed
point of $Z$, and $\hat {\cO}_{X,Z}\otimes_{\cO_P} I_Y/I_Y^2 \to (I_X+I_Y)/(I_X^2+I_Y)$
is a map of a free module onto $I_{Z/X}/I_{Z/X}^2$, so $Q(X,Y)$ is the direct sum of
$\overline Q(Z)$ and a free module, as required.
\end{proof}

\subsection{Decomposing $Q$}

The following result is sometimes useful in computing 
the length of $Q(X,Y)$. Since it reduces
at once to the affine case, we will work with ideals $L,I$ in a Noetherian
ring $A$. We define $Q(L,I)$ to be the cokernel of the map
$$ \Hom(\frac{I+L}{I^2+L}, \frac{A}{I+L}) \to \Hom(\frac{I}{I^2+IL}, \frac{A}{I+L}).
$$
To simplify the notation, if $K\subset A$ then we write
$Q_{A/K}(L,I)$ for $Q((L+K)/K, (I+K)/K)$ computed in the ring $A/K$.

\begin{proposition}\label{reduction1}
\begin{enumerate}
\item\label{transverse1} If $L'\subset L$ is an ideal such that $I\cap L'=IL'$, then
$
Q(L,I)\cong Q_{A/L'}(L,I).
$
In particular, if $I\cap L=IL$, then $Q(L,I)=0$.

\item \label{transverse2}If $I'\subset I$ then $Q_{A/I'}(L,I)\subset Q(L,I)$.

\item\label{transverse3} If $I=I'+I''$ with $I'\cap I''\subset I'I''$ 
and $I'\cap(I''+L)=I'(I''+L)$,
then
$
Q(L,I)= Q_{A/I'}(L,I)+Q_{A/I''}(L,I).
$
\end{enumerate}
\end{proposition}

\begin{proof}
\noindent(1) It is clear that $A/J$ and $(I+L)/(I^2+L)$ remain
the same modulo $L'$. The kernel of the surjection
$I/I^2\to (I+L')/(I^2+L')$ is the image of $I\cap L'$. By hypothesis, 
$I\cap L'\subset IL'\subset IJ$. Since
$\Hom(I/I^2, A/J)=\Hom(I/IJ, A/J)$, we are done.

For parts (2) and (3) we refer to the diagram
$$\begin{diagram}[small]
0&\rTo&Hom(\frac{I'+I^2+L}{I^2+L}, A/J)&\rTo^\beta& Hom(\frac{I'+I^2}{I^2}, A/J) \cr
&&\uTo^\gamma&&\uTo\cr
0&\rTo&Hom(\frac{I+L}{I^2+L}, A/J)&\rTo& Hom(\frac{I}{I^2}, A/J)&\rTo&Q_A(L,I)&\rTo&0  \cr
&&\uTo&&\uTo&&\uTo\alpha\cr
0&\rTo&Hom(\frac{I+L}{I'+I^2+L}, A/J)&\rTo& Hom(\frac{I}{I'+I^2}, A/J)&\rTo&Q_{A/I'}(L,I)&\rTo&0  \cr
&&\uTo&&\uTo&&\uTo\cr
&&0&&0&&0\cr
\end{diagram}
$$
The first two rows, and the first two columns, are obviously exact.
For Part (2) we will show the injectivity of the map labelled $\alpha$.
We first identify
$(I+L)/(I^2+L)$ with $I/(I^2+(I\cap L))$, and
$(I+L)/(I'+ I^2+L)$ with $I/(I'+I^2+(I\cap L))$.
The injectivity of $\alpha$ is thus
equivalent to the statement that that any map $I/I^2\to A/J$ that annihilates both
the images of $I'$ and  of $L$ annihilates the image of $I'+L$, proving (2).

If the hypotheses of Part (3) are satisfied, then 
$$
\frac{I}{I^2} = 
\frac{I'+I^2}{I^2} \oplus \frac{I''+I^2}{I^2},
$$
 and similarly modulo $L$. Thus both the left hand
vertical sequences in the diagram are split exact, and we deduce
both that
$Q_A(L,I) = Q_{A/I'}(L,I) \oplus \coker \beta$,
and that $\coker \beta = Q_{A/I''}(L,I)$.
\end{proof}

\newcommand{\closer}{\vspace{-1.5ex}}

\bigskip
\vbox{\noindent Author Addresses:\par
\smallskip
\noindent{Roya Beheshti}\par
\noindent{Department of Mathematics, Washington University, St. Louis,
MO 63130}\par
\noindent{beheshti@math.wustl.edu}\par
\smallskip
\noindent{David Eisenbud}\par
\noindent{Department of Mathematics, University of California, Berkeley,
Berkeley CA 94720}\par
\noindent{eisenbud@math.berkeley.edu}\par
}

\end{document}